\documentclass[a4paper,pdftex,reqno]{amsart}

\usepackage{geometry}
\title{A generalization of the center theorem of the Thurston-Wolpert-Goldman Lie algebra}
\author{Aoi Wakuda}

\makeatletter
\@namedef{subjclassname@2020}{\textup{2020} Mathematics Subject Classification}
\makeatother

\address{GRADUATE SCHOOL OF MATHEMATICAL SCIENCES, UNIVERSITY OF TOKYO, 3-8-1 KOMABA, MEGURO-KU, TOKYO, 153-8914,
JAPAN}
\email{aoichan19991226@g.ecc.u-tokyo.ac.jp}

\usepackage{amsmath} 
\usepackage{amssymb} 
\usepackage{amsfonts} 
\usepackage{mathrsfs} 
\usepackage{amsthm} 
\usepackage{bm} 
\usepackage[new]{old-arrows} 
\usepackage[dvipdfmx]{graphicx} 
\usepackage{wrapfig} 
\usepackage{color} 
\usepackage[labelsep=colon]{caption} 
\usepackage[labelformat=empty,subrefformat=parens]{subcaption} 
\usepackage{multicol} 
\usepackage{units} 
\usepackage[all]{xy} 
\usepackage{tikz-cd} 
\usepackage{hyperref} 
\usepackage{color} 
\usepackage{undertilde} 
\usepackage{latexsym}
\usepackage{mathtools}
\usepackage{comment}
\usepackage{tikz-cd}
\usepackage{tikz}
\usepackage{float}
\usetikzlibrary{decorations.markings}
\usetikzlibrary{arrows,arrows.meta}
\usepackage{graphicx}

\newtheorem{thm}{Theorem}[section]
\newtheorem{prop}[thm]{Proposition}
\newtheorem{lemma}[thm]{Lemma}

\newtheorem*{mthm*}{Main Theorem}

\theoremstyle{definition}
\newtheorem{defi}[thm]{Definition}
\newtheorem{ex}[thm]{Example}
\newtheorem{remark}[thm]{Remark}

\newcommand{\Z}{\mathbb{Z}}
\newcommand{\R}{\mathbb{R}}
\newcommand{\C}{\mathbb{C}}

\newcommand{\wt}{\widetilde}
\newcommand{\ut}{\utilde}
\newcommand{\dfr}{\displaystyle\frac}
\newcommand{\Al}{\alpha}
\newcommand{\Be}{\beta}
\newcommand{\Ga}{\gamma}
\newcommand{\De}{\delta}
\newcommand{\La}{\lambda}

\newcommand{\wampbiz}{\wt{(\alpha^m*_{P}\beta_i)}_{0}}
\newcommand{\wampbii}{\wt{(\alpha^m*_{P}\beta_i)}_\infty}
\newcommand{\uampbiz}{\utilde{(\alpha^m*_{P}\beta_i)}_{0}}
\newcommand{\uampbii}{\utilde{(\alpha^m*_{P}\beta_i)}_\infty}

\begin{document}

\begin{abstract}

The Goldman Lie algebra of an oriented surface was defined by Goldman \cite{Goldman}. By the natural involution that opposes the orientation of curves, the Goldman Lie algebra becomes a $\mathbb{Z}_2$-graded Lie algebra. Its even part is isomorphic to the Thurston-Wolpert-Goldman Lie algebra or, briefly, the TWG Lie algebra. Chas and Kabiraj \cite{Chas-Kabiraj} proved the center of the TWG-Lie algebra is generated by the class of the unoriented trivial loop and the classes of unoriented loops parallel to boundary components or punctures. The center of the even part can be rephrased as the set of elements of the even part annihilated by all the elements of the even part. We also prove some similar statements for the remaining 3 cases involving the odd part. Moreover, we compute the elements of the symmetric algebra and the universal enveloping algebra of the Goldman Lie algebra annihilated by all the even elements of the Goldman Lie algebra, and those annilated by all the odd elements.

\end{abstract}

\maketitle

\section{Introduction} \label{intro}

Let $S$ be an oriented surface (possibly with boundary and punctures). We assume the Euler characteristic of $S$ is negative so that $S$ admits a hyperbolic metric. We also assume that the interior of $S$ is not homeomorphic to that of a pair of pants. We call it {\it a surface of type of a pair of pants.} Denote by $\hat{\pi}$ (resp. $\wt{\pi}$) the set of free homotopy classes of directed (resp. undirected) closed curves on $S$. Unless otherwise specified, we assume $K$ to be a commutative ring containing the ring $\Z\bigl[\dfr{1}{2}\bigr]$ . For any set $X$, we denote by $KX$ the free $K$-module generated by $X$.

In 1986, Goldman \cite{Goldman} introduced a Lie algebra structure on $K\hat{\pi}$ in his study of the symplectic structure on the moduli space of representations of the fundamental group of $S$. This bracket is called the {\it Goldman bracket}, and denoted by $[-,-]_{G}$. The free $K$-module $K\hat{\pi}$ with the Goldman bracket is called the {\it Goldman Lie algebra}. 

There is a natural involution $\iota : \hat{\pi} \to \hat{\pi}$ given by $\alpha\mapsto \alpha^{-1}$, which maps the curve $\alpha$ to the curve $\alpha$ with opposite orientation. By extending $\iota$ linearly to $K\hat{\pi}$, we obtain a Lie algebra involution $\iota : K\hat{\pi} \to K\hat{\pi}$. The involution $\iota$ defines two submodules of the Goldman Lie algebra of $S$. Recall $K$ is a commutative ring including $\Z\bigl[\dfr{1}{2}\bigr]$. Let $A_{0}$ be a submodule of $K$$\hat{\pi}$ generated by the elements of the form $\alpha+\iota\alpha$ and {\it$A_{1}$} a submodule of $K$$\hat{\pi}$ generated by the elements of the form $\alpha-\iota\alpha$. The Goldman Lie algebra $K\hat{\pi}$ has the decomposition
\begin{align}
K\hat{\pi}=A_{0}\oplus A_{1}.
\end{align}

\noindent
Since $\iota[\Al,\Be]_{G}=[\iota\Al,\iota\Be]_{G}$ for $\Al, \Be \in K\hat{\pi}$, the submodule $A_{0}$ is a Lie subalgebra of $K\hat{\pi}$. 
By the isomorphism $A_{0} \to K\wt{\pi}$ given by $\Al+\iota\Al\mapsto \wt{\Al}$ where $\widetilde{\alpha}$ denotes the free homotopy class $\alpha$ forgotten its orientation, the submodule $A_{0}$ is isomorphic to $K\wt{\pi}$ as a $K$-module. Therefore, $K\wt{\pi}$ becomes a Lie algebra induced by the isomorphism from $A_{0}$ to $K\wt{\pi}$. Following to Chas and Kabiraj \cite{Chas-Kabiraj}, we call the Lie algebra $K\wt{\pi}$ {\it Thurston-Wolpert-Goldman Lie algebra} or, briefly, the {\it TWG Lie algebra}.


One of the reasons we started thinking about $A_{1}$ is the Turaev cobracket  \cite{Turaev1991}. 
\begin{align*}
\De_{T} : (A_{0}/{K{\bf{1}}})\oplus A_{1} \to ((A_{0}/{K{\bf{1}}})\oplus A_{1})\otimes ((A_{0}/{K{\bf{1}}})\oplus A_{1}),
\end{align*}

\noindent
where $\bf{1}$ is the class of the trivial loop. The Goldman bracket together with the Turaev cobracket makes $K\hat{\pi}/K\bf{1}$ Lie bialgebra \cite{Turaev1991}. We came up with the following natural question: Does the TWG bracket together with the cobracket coming from the Turaev cobracket make $K\wt{\pi}/K\wt{\bf{1}}$ Lie bialgebra ?

The answer is no. By the definition of the Turaev cobracket, we have
\begin{align*}
\De_{T}(A_{0}/{K{\bf{1}}}) \subset ((A_{0}/{K{\bf{1}}})\otimes A_{1})\oplus (A_{1}\otimes (A_{0}/{K{\bf{1}}})).
\end{align*}
In the paper, we prove that $A_{1}$ has similar properties as $A_{0}$, though we don't discuss further the Turaev cobracket.

Computing the center of a given Lie algebra is a fundamental problem. We also call a closed curve {\it non-essential} if it is homotopic to a point or a boundary curve or to a puncture. 

In \cite{Etingof2006}, Etingof proved that the center of the Goldman Lie algebra $\C\hat{\pi}$ on a closed surface is generated by {\bf{1}}. He used moduli spaces of flat bundles to compute the center of the Goldman Lie algebra. In \cite{Kawazumi-Kuno2013}, Kawazumi and Kuno proved using representation theory of the symplectic group that the center of the Goldman Lie algebra on a surface with infinite genus and one boundary component is generated by the class of non-essential curves. In \cite{Kabiraj2016}, Kabiraj proved that the center of the Goldman Lie algebra on surfaces with boundary except surfaces of type of a pair of pants is generated by the class of non-essential curves. His proof introduced hyperbolic geometry to study the center of the Goldman Lie algebra. 
\begin{thm}
\cite[Center Theorem]{Chas-Kabiraj} 
The center of the TWG Lie algebra is generated by the class of non-essential curves as a $K$-module.
\end{thm}

\noindent
For any $K$-submodule $A$ of $K\hat{\pi}$ and $K$-submodule $M$ of a $K\hat{\pi}$-module, we denote
\begin{center}
$\mathrm{Ann}_{M}{A}\coloneq\lbrace m\in M| [a, m] = 0$ for all $a\in A\rbrace$.
\end{center}
The Center Theorem can be rephrased as determing the set $\mathrm{Ann}_{A_{0}}{A_{0}}$. In this paper, we address the following natural question: how about the set of $\mathrm{Ann}_{A_{i}}{A_{j}}$ in the case $(i,j)=(0,1),(1,0),(1,1)$ ?
Our main result is to give an answer to this question. 

\begin{thm} [Annihilator Theorem] \label{Annihilator Theorem}
The annihilator $\mathrm{Ann}_{A_0}{A_1}$ of $A_1$ in $A_0$ is generated by the elements of the form $\Al+\iota\Al$ such that $\Al$ is non-essential. The annihilator $\mathrm{Ann}_{A_1}{A_i}$ of $A_i$ in $A_1$  $(i=0,1)$ is generated by the elements of the form $\Al-\iota\Al$ such that $\Al$ is non-essential. In other words, the annihilator $\mathrm{Ann}_{K\hat{\pi}}{A_i}$ of $A_i$ in $K\hat{\pi}$ $(i=0,1)$ is generated by the classes of non-essential curves as a $K$-module.
\end{thm}
\noindent
In order to prove this result, we use the technique of hyperbolic geometry introduced by Chas and Kabiraj in \cite{Chas-Kabiraj}. Theorems similar to the Annihilator Theorem hold for the universal enveloping algebra $U(K\hat{\pi})$ and the symmetric algebra $S(K\hat{\pi})$  of $K\hat{\pi}$. 

\begin{thm} \label{The Annihilator Theorem for the symmetric algebra}
The annihilator $\mathrm{Ann}_{S(K\hat{\pi})}{A_i}$ of $A_i$ in $S(K\hat{\pi})$ $(i=0,1)$ is generated by scalars $K$, the classes of non-essential curves as a $K$-algebra.
\end{thm}

The above Theorem \ref{The Annihilator Theorem for the symmetric algebra} naturally extends to the Poisson algebras $S_{k}(K\hat{\pi})$ which are the deformations of $S(K\hat{\pi})$ introduced by Turaev in \cite[Section 2.2]{Turaev1991}
\begin{thm}
The annihilator $\mathrm{Ann}_{U(K\hat{\pi})}{A_i}$ of $A_i$ in $U(K\hat{\pi})$ $(i=0,1)$ is generated by scalars $K$, the classes of non-essential curves as a $K$-algebra.
\end{thm}

\noindent
\textbf{Organization of the paper.} 
In Section 2, we recall the definitions and properties of the Goldman Lie algebra introduced in \cite{Goldman} and the TWG Lie algebra introduced in \cite{Chas-Kabiraj}. In Section 3, we recall some results between the hyperbolic geometry of the upper half plane $\mathbb{H}$ and the TWG bracket. We prove the Annihilator Theorem in Section 4. In the final Section 5, we compute the annihilator $\mathrm{Ann}_{S(K\hat{\pi})}{A_i}$ of $A_i$ in $S(K\hat{\pi})$ and the annihilator $\mathrm{Ann}_{U(K\hat{\pi})}{A_i}$ of $A_i$ in $U(K\hat{\pi})$ for $i=0,1$.\\

\noindent
\textbf{Acknowledgment.} The author would like to thank my supervisor, Nariya Kawazumi, for many discussions and helpful advice. The author would also like to thank Kento Sakai for noticing Lemma \ref{Lem 3.6} does not hold in the case where $S$ is a surface of type of a pair of pants. This work was supported by the WINGS-FMSP program at the Graduate School of Mathematical Sciences, the University of Tokyo.

\section{The Goldman Lie algebra and The Thurston-Wolpert-Goldman Lie algebra} \label{Goldman Lie algebra and Thurston-Wolpert-Goldman Lie algebra}

In this section, we recall the definition of the Goldman Lie algebra and the TWG Lie algebra. 

\subsection{The Goldman Lie algebra} \label{The Goldman Lie algebra}

Let $S$ be an oriented surface and $K$ a commutative ring including $\Z\bigl[\dfr{1}{2}\bigr]$. Denote by $\hat{\pi}$ the set of free homotopy classes of directed closed curves on $S$ and by $|{\alpha}|$ the free homotopy class of a directed closed curve $\alpha$.

The {\it Goldman bracket} of $\alpha$, $\beta$ $\in$ $\hat{\pi}$ is defined by the formula
\begin{align}
[\alpha,\beta]_{G}\coloneqq\sum_{P\in{\alpha}\cap{\beta}}\varepsilon_{P}(\alpha,\beta)|{\alpha_{P}\beta_{P}}| .
\end{align}

Here the representatives $\alpha$ and $\beta$ are chosen so that they intersect transversely in a set
of double points ${\alpha}\cap{\beta}$ , $\varepsilon_{P}(\alpha,\beta)$ denotes the sign of the intersection between $\alpha$ and $\beta$ at an
intersection point ${P}$, and ${\alpha_{P}\beta_{P}}$ denotes the loop product of $\alpha$ and $\beta$ at $P$.\\

In \cite{Goldman}, Goldman proved the bracket defined above is well defined, skew-symmetric, and satisfies the Jacobi identity on $K$$\hat{\pi}$. In other words, $K$$\hat{\pi}$ is a Lie algebra.

\subsection{The Thurston-Wolpert-Goldman Lie algebra} \label{The Thurston-Wolpert-Goldman Lie algebra}

Let $\sim$ be the equivalence relation on $\hat{\pi}$, such that $\Al$ $\sim$ $\Be$ if and only if $\Al = \Be$ or $\Al$ = $\Be^{-1}$. Denote by $\widetilde{\pi}$ the quotient set of $\hat{\pi}$ by $\sim$. Geometrically, we can interpret $\widetilde{\pi}$ as the set of free homotopy classes of undirected closed curves on $S$. There is a natural surjection $\widetilde{\cdot} : \hat{\pi} \to \widetilde{\pi}$ ; $\alpha$ $\mapsto$ $\widetilde{\alpha}$ where $\widetilde{\alpha}$ denotes the free homotopy class $\alpha$ forgotten its orientation. By extending $\widetilde{\cdot}$ linearly to $K$$\hat{\pi}$, we obtain a surjective $K$-linear map $\widetilde{\cdot} : K\hat{\pi} \to K\widetilde{\pi}$. \\
\begin{defi}
Given two directed curves $\alpha$ and $\beta$ on $S$ and a transversal intersection point P of
$\alpha$ and $\beta$, we define new curves $(\alpha*_{P}\beta)_{0},  (\alpha*_{P}\beta)_{\infty}$
\begin{align}
(\alpha*_{P}\beta)_{0}\coloneqq |{\alpha_{P}\beta^{\varepsilon_{P}(\alpha,\beta)}_{P}}|, \\
(\alpha*_{P}\beta)_{\infty}\coloneqq |{\alpha_{P}\beta^{-\varepsilon_{P}(\alpha,\beta)}_{P}}|.
\end{align}
\noindent
The new curve $(\alpha*_{P}\beta)_{0}$ (resp.$(\alpha*_{P}\beta)_{\infty}$) means modifying the orientation of $\beta$ so that the sign of the intersection between $\alpha$ and $\beta$ at an intersection point $P$ is 1 (resp.$-1$).

\end{defi}

In \cite{Chas-Kabiraj}, $(\alpha*_{P}\beta)_{0} $ and $(\alpha*_{P}\beta)_{\infty}$ are defined in the case where $\Al$ and $\Be$ are undirected. In this paper, we define them in the case where $\Al$ and $\Be$ are directed. Therefore, after taking the upper tilde, this notation matches that of \cite{Chas-Kabiraj}.

\begin{prop}
The {\it TWG bracket} of $\wt{\Al}$, $\wt{\Be}$ $\in$ $\wt{\pi}$ can be written as:
\begin{align}
[\widetilde{\alpha},\widetilde{\beta}]_{TWG}=\sum_{P\in{\alpha}\cap{\beta}}\widetilde{(\alpha*_{P}\beta)}_{0}-\widetilde{(\alpha*_{P}\beta)}_{\infty}.
\end{align}

\noindent
Here the representatives $\alpha$ and $\beta$ are chosen so that they intersect transversely in a set of double points $\alpha\cap\beta$.

\end{prop}

\noindent
By the decomposition $K\hat{\pi}=A_{0}\oplus A_{1}$, we can identify $A_{0}$ with $K\hat\pi/A_{1}$. The surjection $\wt{\cdot}$ induces an isomorphism from $K\hat\pi/A_{1}$ to $K\wt{\pi}$. On the other hand, we denote by $\ut{\cdot}$ the natural surjection from $K\hat{\pi}$ to $K\hat{\pi}/A_{0}$.
Also, we have a $K$-module isomorphism from $A_{1}$ to $K\hat{\pi}\slash{A_{0}}$ given by $\alpha-\iota\alpha$ to $\ut\alpha$.
To align with the notation above, we write $K\ut{\pi}$ instead of $K\hat{\pi}\slash{A_{0}}$.


Note that a curve $\Al$ is identified with the curve $\iota\Al$ in $K\widetilde{\pi}$, while a curve $\Al$ is identified with minus the curve $-\iota\Al$ in $K\ut{\pi}$.

The bracket on $K\wt{\pi}$ coming from the Goldman bracket on $A_0$ is the TWG bracket. Similarly, we call it the TWG bracket that the bracket on $K\wt\pi\oplus K\ut\pi$ induced by the Goldman bracket on $K\hat{\pi}$. Therefore, the following holds:

\begin{prop}
For $\Al, \Be \in \hat{\pi}$, the {\it TWG bracket} $[\wt{\alpha},\ut{\beta}]_{TWG}$, $[\ut{\alpha},\wt{\beta}]_{TWG}$ and $[\ut{\alpha},\ut{\beta}]_{TWG}$ can be written as:
\begin{align}
&[\wt{\alpha},\ut{\beta}]_{TWG}=\sum_{P\in{\alpha}\cap{\beta}}\varepsilon_{P}(\Al,\Be)(\ut{(\alpha*_{P}\beta)}_{0}+\ut{(\alpha*_{P}\beta)}_{\infty}),\\
&[\ut{\alpha},\wt{\beta}]_{TWG}=\sum_{P\in{\alpha}\cap{\beta}}\ut{(\alpha*_{P}\beta)}_{0}-\ut{(\alpha*_{P}\beta)}_{\infty},\\
&[\ut{\alpha},\ut{\beta}]_{TWG}=\sum_{P\in{\alpha}\cap{\beta}}\varepsilon_{P}(\Al,\Be)(\widetilde{(\alpha*_{P}\beta)}_{0}+\widetilde{(\alpha*_{P}\beta)}_{\infty}).
\end{align}

\noindent
Here the representatives $\alpha$ and $\beta$ are chosen so that they intersect transversely in a set of double points ${\alpha}\cap{\beta}$.
\end{prop}

\section{Hyperbolic geometry and the Thurston-Wolpert-Goldman Lie algebra} \label{Hyperbolic geometry and the Thurston-Wolpert-Goldman Lie algebra}

In this section, we recall the Center Theorem of the Thurston-Wolpert-Goldman Lie algebra, due to Chas and Kabiraj \cite{Chas-Kabiraj}.
Denote by $Teich(S)$ the Teichm\"{u}ller space of the hyperbolic surface $S$. For a free homotopy class $\Al, \Al(X)$ denotes a geodesic representative of $\Al$ with respect to the hyperbolic metric $X\in Teich(S)$.

\subsection{Geometric intersection number for geodesic. }
In order to distinguish intersection points of multiplicity larger than two, we introduce another definition of
intersection point: If $a$ and $b$ are two closed curves intersecting transversally, an $(a, b)$-{\it intersection point} is a point $P$ on the intersection of $a$ and $b$, together with a choice of a pair of small arcs, one of $a$ and the other of $b$, intersecting only at $P$.

\begin{lemma} 
\cite[Lemma 2.2]{Chas-Kabiraj}
If $X_{1}, X_{2}\in Teich(S)$ are hyperbolic metrics on $S$ and $\Al, \Be \in \hat{\pi}$ are two free homotopy classes, then there is a canonical one to one correspondence between $(\Al(X_{1}), \Be(X_{1}))$-{\it intersection points} and $(\Al(X_{2}), \Be(X_{2}))$-{\it intersection points}. 
\end{lemma}

\subsection{Angles and earthquakes. }

For each $(\Al(X), \Be(X))$-intersection point $P$, {\it the X angle of $\Al$ and $\Be$ at $P$}, denoted by $\theta_{P}(X)$, is defined as the angle of $\Al(X)$ and $\Be(X)$ at $P$, measured from $\Be(X)$ to $\Al(X)$ by using the orientation of $S$.

\begin{figure}[H]
\label{Angle}
\centering\includegraphics[width=0.4\linewidth]{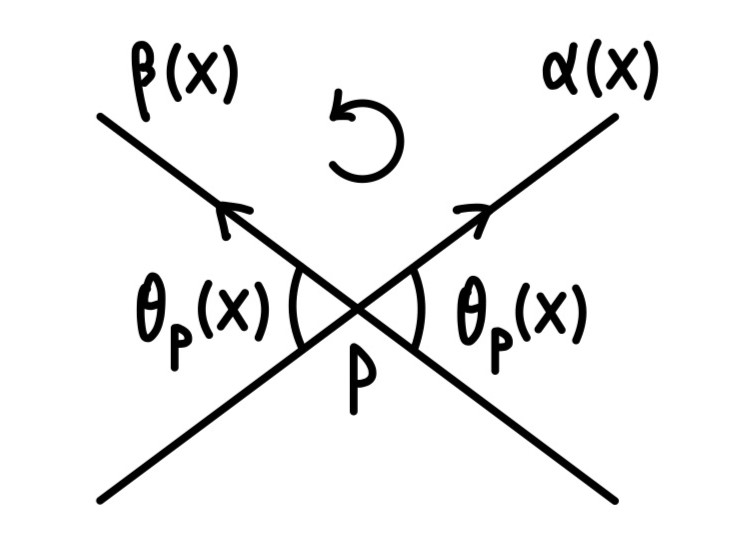}
\caption{\label{fig.angle_by_metric}The angle $\theta_{P}(X)$}
\end{figure}

\noindent
In particular, since $P$ is a transversal intersection point, $\theta_{P}(X)$ $\in$ $(0,\pi)$.
Following \cite{Kerckhoff1983} for each $X$ $\in$ $Teich(S)$ and each simple $\Al$ $\in$ $\hat{\pi}$ and each real number $t$, $S_{\Al(X)}(t)$ $\in$ $Teich(S)$ is given by left twist deformation of $X$ along $\Al(X)$ at the time $t$ starting at $X$. By \cite[Lemma 2.1]{Kabiraj2018} we get

\begin{lemma} \label{lem twist}
\cite[Lemma 2.4]{Chas-Kabiraj}
If $X$ $\in$ $Teich(S)$ and $\Al$, $\Be$ $\in$ $\hat{\pi}$ such that $\Al$ is simple, and $P$ is an $(\Al(X), \Be(X))$-intersection point, then $\theta_{P}(S_{\Al(X)}(t))$ $\in$ $Teich(S)$ is a strictly decreasing function with respect to $t$.
\end{lemma}
 
For each hyperbolic isometries $g$ of the hyperbolic plane, let $A_{g}$ be the axis of $g$ and $t_{g}$ the translation length of $g$.

\begin{thm} \label{thm cosh}
\cite[Theorem 7.38.6]{Beardon1983}
Let $g$ and $h$ be hyperbolic isometries of the hyperbolic plane and suppose that $A_{g}$ and $A_{h}$ intersect at a point $P$. Denote by $\theta$ the angle at P between forward direction of $A_{g}$ and $A_{h}$. Then the product $gh$ is hyperbolic and 
\begin{align}
\cosh\left(\frac{t_{gh}}{2}\right) = \cosh\left(\frac{t_{g}}{2}\right)\cosh\left(\frac{t_{h}}{2}\right) + \sinh\left(\frac{t_{g}}{2}\right)\sinh\left(\frac{t_{h}}{2}\right)\cos\theta.
\end{align}

\begin{figure}[H]
\centering\includegraphics[width=0.5\linewidth]{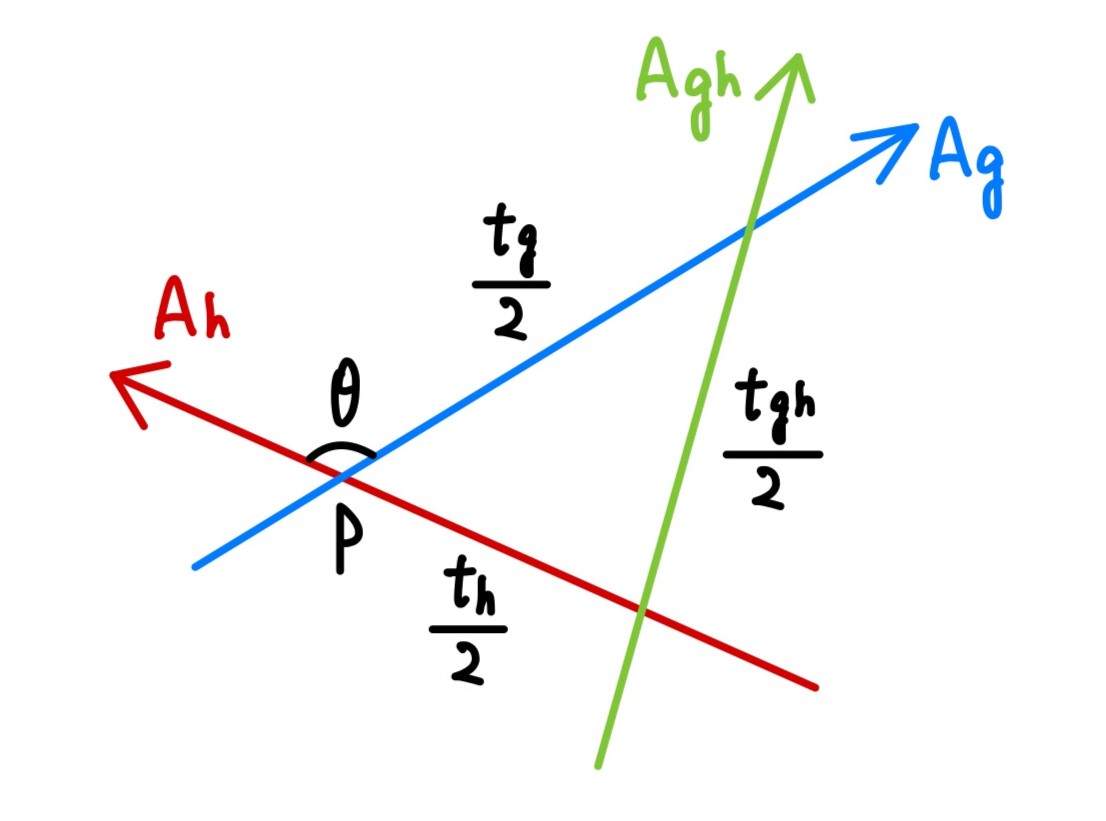}
\caption{\label{fig.cosh}Theorem \ref{thm cosh}}
\end{figure}

\end{thm}

For each hyperbolic metric $X$ and each $\Al$ $\in$ $\hat{\pi}$, denote by $\ell_{\Al(X)}$ the length of $\Al(X)$.
By using this theorem, the next lemma follows immediately.

\begin{lemma} \label{Lem 2.7}
\cite[Lemma 2.7]{Chas-Kabiraj}
Let $X$ be a hyperbolic metric on $S$ and $\Al$, $\Be$ $\in$ $\hat{\pi}$ be free homotopy classes. If $P$ is an ($\Al(X)$, $\Be(X)$)-intersection point, then for any metric $Y$ $\in$ $Teich(S)$, we have 
\begin{align*}
\cosh\left(\frac{\ell_{(\alpha*_{P}\beta)_{0}(Y)}}{2}\right) = \cosh\left(\frac{\ell_{\Al(Y)}}{2}\right)\cosh\left(\frac{\ell_{\Be(Y)}}{2}\right) - \sinh\left(\frac{\ell_{\Al(Y)}}{2}\right)\sinh\left(\frac{\ell_{\Be(Y)}}{2}\right)\cos\theta_{P}(Y),\\
\cosh\left(\frac{\ell_{(\alpha*_{P}\beta)_{\infty}(Y)}}{2}\right) = \cosh\left(\frac{\ell_{\Al(Y)}}{2}\right)\cosh\left(\frac{\ell_{\Be(Y)}}{2}\right) + \sinh\left(\frac{\ell_{\Al(Y)}}{2}\right)\sinh\left(\frac{\ell_{\Be(Y)}}{2}\right)\cos\theta_{P}(Y).
\end{align*}

\end{lemma}

The angles at $P$ of $\Al(X)$ and $\Be(X)$ and of $\Al^m(X)$ and $\Be(X)$ are the same angle. Therefore, we consider $P$ as an $(\Al^m, \Be)$-intersection point, and the angle $\theta_P(X)$ also denotes the angle at $P$ measured from $\Be(X)$ to $\Al^m(X)$.

\begin{lemma} \label{Prop 3.4}
\cite[Proposition 3.4]{Chas-Kabiraj}
Let $X \in Teich(S)$ be a hyperbolic metric on $S$ and $\Al, \Be_{1}, \Be_{2} \in \hat{\pi}$ such that $\Al$ is simple and $\Al(X)$ and $\Be_{i}(X)$ are intersecting transversally for $i=1,2$. Let $P$ and $Q$ be $(\Al(X), \Be_{1}(X))$ and $(\Al(X), \Be_{2}(X))$-intersection points respectively. The following holds:

(1) The equality $\wt{(\alpha^m*_{P}\Be_{1})}_{0}=\wt{(\alpha^m*_{Q}\Be_{2})}_\infty$ holds for at most one positive value of $m$.

(2) Either the equality $\wt{(\alpha^m*_{P}\Be_{1})}_{0}=\wt{(\alpha^m*_{Q}\Be_{2})}_{0}$ (resp. $\wt{(\alpha^m*_{P}\Be_{1})}_\infty=\wt{(\alpha^m*_{Q}\Be_{2})}_\infty$) holds for at most one positive value of $m$, or there exists an $(\Al(X), \Be_{1}(X))$-intersection point $R$ such that $\theta_P(X) < \theta_R(X)$.
\end{lemma}

\begin{lemma} \label{Lem for type3}
Let $X \in Teich(S)$ be a hyperbolic metric on $S$ and $\Al, \Be, \Ga \in \hat{\pi}$ be free homotopy classes such that $\Al(X)$ and $\Be(X)$ are intersecting transversally. If $P$ is an $(\Al(X), \Be(X))$-intersection point, then the equality $\wt{(\alpha^m*_{P}\beta)}_{0}=\wt{\Ga}$ (resp. $\wt{(\alpha^m*_{P}\beta)}_{\infty}=\wt{\Ga}$) holds for at most two positive values of $m$.
\end{lemma}

\begin{proof}
We argue by contradiction. Suppose that there are three distinct positive integers $m_{1}<m_{2}<m_{3}$ such that 
\begin{align*}
\wt{(\alpha^{m_{1}}*_{P}\beta)}_{0}=\wt{(\alpha^{m_{2}}*_{P}\beta)}_{0}=\wt{(\alpha^{m_{3}}*_{P}\beta)}_{0}=\wt{\Ga}.
\end{align*}
\noindent
To simplify notation, we will not write the dependence on $X$ (for example, we
will write $\cos\theta_{P}$ instead of $\cos\theta_{P}(X)$)
By Lemma \ref{Lem 2.7}, we have
\begin{align*}
&\cosh\left(\frac{\ell_{\Al^{m_{1}}}}{2}\right)\cosh\left(\frac{\ell_{\Be}}{2}\right) - \sinh\left(\frac{\ell_{\Al^{m_{1}}}}{2}\right)\sinh\left(\frac{\ell_{\Be}}{2}\right)\cos\theta_{P}\\
=&\cosh\left(\frac{\ell_{\Al^{m_{2}}}}{2}\right)\cosh\left(\frac{\ell_{\Be}}{2}\right) - \sinh\left(\frac{\ell_{\Al^{m_{2}}}}{2}\right)\sinh\left(\frac{\ell_{\Be}}{2}\right)\cos\theta_{P}\\
=&\cosh\left(\frac{\ell_{\Al^{m_{3}}}}{2}\right)\cosh\left(\frac{\ell_{\Be}}{2}\right) - \sinh\left(\frac{\ell_{\Al^{m_{3}}}}{2}\right)\sinh\left(\frac{\ell_{\Be}}{2}\right)\cos\theta_{P}.
\end{align*}

On the other hand, for any real number $c$, the equation
\begin{align*}
\cosh\left(x\right)\cosh\left(\frac{\ell_{\Be}}{2}\right) - \sinh\left(x\right)\sinh\left(\frac{\ell_{\Be}}{2}\right)\cos\theta_{P}=c
\end{align*}
\noindent
has at most two positive real solutions. Since $\ell_{\Al^{m_{1}}}<\ell_{\Al^{m_{2}}}<\ell_{\Al^{m_{3}}}$, we get a contradiction.
\end{proof}

\begin{lemma} \label{Lem for type4}
Let $X \in Teich(S)$ be a hyperbolic metric on $S$ and $\Al, \Be \in \hat{\pi}$ be free homotopy classes such that $\Al(X)$ and $\Be(X)$ are intersecting transversally. If $P$ is an $(\Al(X), \Be(X))$-intersection point, then the equality $\wt{(\alpha^m*_{P}\beta)}_{0}=\wt{\alpha^m}$ (resp. $\wt{(\alpha^m*_{P}\beta)}_{\infty}=\wt{\alpha^m}$) holds for at most one positive value of $m$.
\end{lemma}

\begin{proof}
Suppose $\wt{(\alpha^m*_{P}\beta)}_{0}=\wt{\alpha^m}$ for two distinct positive values of $m$. By Lemma \ref{Lem 2.7}, we have
\begin{align*}
\cosh\left(\frac{\ell_{\Al^{m}}}{2}\right)\cosh\left(\frac{\ell_{\Be}}{2}\right) - \sinh\left(\frac{\ell_{\Al^{m}}}{2}\right)\sinh\left(\frac{\ell_{\Be}}{2}\right)\cos\theta_{P}=\cosh\left(\frac{\ell_{\Al^{m}}}{2}\right)
\end{align*}
\noindent
This implies
\begin{align*}
\cosh\left(\frac{\ell_{\Be}}{2}\right) - \tanh\left(\frac{\ell_{\Al^{m}}}{2}\right)\sinh\left(\frac{\ell_{\Be}}{2}\right)\cos\theta_{P}=1.
\end{align*}
\noindent
If $\theta_{P}=\frac{\pi}{2}$, then $\cosh\left(\frac{\ell_{\Be}}{2}\right)=1$. Hence, $\ell_{\Be}=0$, which contradicts $\ell_{\Be}>0$. 

Therefore, we get
\begin{align} \label{tanh and fraction}
\tanh\left(\frac{\ell_{\Al^{m}}}{2}\right)=\frac{\cosh\left(\frac{\ell_{\Be}}{2}\right)-1}{\sinh\left(\frac{\ell_{\Be}}{2}\right)\cos\theta_{P}}
\end{align}
Note that the right-hand side of (\ref{tanh and fraction}) does not depend on $m$. Since the function $\tanh(x)$ is a strictly increasing function of $x$, the equation (\ref{tanh and fraction}) contradicts the hypothesis.
\end{proof}

For $\Al, \Be \in \hat{\pi}$, let $i(\Al, \Be)$ be the geometric intersection number of $\Al$ and $\Be$.

\begin{lemma} \label{Lem 3.5}
\cite[Lemma 3.5]{Chas-Kabiraj}
Let $\Al, \Be_1, \ldots, \Be_k, \iota\Al, \iota\Be_1, \ldots, \iota\Be_k$ be pairwise distinct free homotopy classes of closed curves such that $\Al$ is simple. Let $\Be = \sum_{j=1}^{k} c_j \Be_j$ where the coefficients $c_1, c_2, \ldots, c_k$ are nonzero elements of $K$. Then either $i(\Al, \Be_j) = 0$ for all $j \in \{1, \ldots, k\}$ or there exists a positive integer $m_0$ such that $[\wt{\Al^m}, \wt{\Be}]_{TWG} \neq 0$ for all $m \geq m_0$.
\end{lemma}



\begin{lemma} \label{Lem 3.6}
\cite[Lemma 3.6]{Chas-Kabiraj}
If $S$ is an orientable surface and $\Be$ is a closed curve on $S$ such that $i(\Al, \Be) = 0$ for every simple closed curve $\Al$. Then $\Be$ is non-essential.
\end{lemma}

\begin{remark} 
Note that Lemma \ref{Lem 3.6} dose not hold in the case where $S$ is a surface of type of a pair of pants.

\end{remark}

\noindent
\textbf{Center Theorem.} \cite{Chas-Kabiraj}
{\it The center of the TWG Lie algebra is generated by the class of non-essential curves as a $K$-module.}

\section{proof of the annihilator theorem}
In \cite{Chas-Kabiraj}, Chas and Kabiraj determined that the center of the TWG Lie algebra is generated by the class of curves homotopic to a point, and the classes of curves winding multiple times around a single puncture or boundary component. This theorem can be rephrased as determing the generators of the annihilator of $A_0$ in $A_0$. In this section, we study the annihilator of $A_i$ in $A_j$ in the case where $(i,j)=(0,1),(1,0),(1,1)$.

\begin{lemma} \label{Geodesic and 1 Lem}
Let $\Al$ be an element of $\hat{\pi}$. Then the following conditions are equivalent.
 
\begin{enumerate}
  \item[(1)] $\Al=\iota{\Al}$
  \item[(2)] $\Al=\bf{1}$
\end{enumerate}
where $\bf{1}$ is the class of the trivial loop.

\end{lemma}

\begin{proof}
($(1)\Rightarrow (2)$) : Denote by $\pi_1(S)$ the fundamental group of $S$. There is an element $a \in \pi_1(S)$ such that a representative of $\Al$. Since $\Al=\iota{\Al}$, there exists an element $b \in \pi_1(S)$ such that $bab^{-1}=a^{-1}$. Denote by $\varphi : \pi_1(S) \to PSL(2, \mathbb{R})$ a natural injective group homomorphism induced by a hyperbolic metric $X\in Teich(S)$. There are some real numbers $p, q, r, s \in \mathbb{R}$ with $ps-qr=1$ such that $\varphi(b) = \pm\begin{pmatrix} p & q \\ r & s \end{pmatrix}$.

\underline{Case (A)} : The element $\varphi(a) \in PSL(2, \mathbb{R})$ is parabolic. 

\noindent
There exists a nonzero real number $t$ such that $\varphi(a) = \pm\begin{pmatrix} 1 & t \\ 0 & 1 \end{pmatrix}$. We have
\begin{align} \label{ba=a-1b}
\varphi(b)\varphi(a)=\varphi(a)^{-1}\varphi(b).
\end{align}
This implies 

$$
\pm\begin{pmatrix} p & pt+q \\ r & rt+s \end{pmatrix} = \pm\begin{pmatrix} p & -pt+q \\ r & -rt+s \end{pmatrix}.
$$
Hence we obtain $p=r=0$, however $ps-qr=1$. This is a contradiction.

\underline{Case (B)} : The element $\varphi(a) \in PSL(2, \mathbb{R})$ is hyperbolic. 

\noindent
There exists a real number $\La$ such that $\La>1$ and $\varphi(a) = \pm\begin{pmatrix} \La & 0 \\ 0 & \La^{-1} \end{pmatrix}$. The equation (\ref{ba=a-1b}) implies

$$
\pm\begin{pmatrix} p\La & q\La^{-1} \\ r\La & s\La^{-1} \end{pmatrix} = \pm\begin{pmatrix} p\La^{-1} & q\La^{-1} \\ r\La & s\La \end{pmatrix}
$$
Hence we obtain $p=s=0$. Since $\operatorname{tr}(\varphi(b))=0$, the element $\varphi(b)\in PSL(2, \mathbb{R})$ is an elliptic element. This is a contradiction.

Therefore, $\varphi(a) \in PSL(2, \mathbb{R})$ is an elliptic element. Then, $\varphi(a) = \pm\begin{pmatrix} 1 & 0 \\ 0 & 1 \end{pmatrix}$. By the injectivity of $\varphi$, $a$ is the identity element of $\pi_1(S)$. Hence we obtain $\Al=\bf{1}$.

\noindent
($(2)\Rightarrow (1)$) : Clearly, $\bf{1}=\iota{\bf{1}}$.
\end{proof}

\begin{lemma} \label{Key Lem}
Let $\Al$ and $\Be$ be elements of $\hat{\pi}$. Then the following conditions are equivalent.
 
\begin{enumerate}
  \item[(1)] $\ut{\Al}=\pm\ut{\Be}$,
  \item[(2)] $\wt{\Al}=\wt{\Be}$.
\end{enumerate}

\end{lemma}
 
\begin{proof}
($(1)\Rightarrow (2)$) : (1) is equivalent to the condition $\Al-\iota\Al=\pm(\Be-\iota\Be)$. In the case where $\Al=\bf{1}$, $\Be=\iota\Be$ since the left hand side is 0. By Lemma \ref{Geodesic and 1 Lem}, $\Be=\bf{1}=\Al$. In the case where $\Al\neq\bf{1}$,  By Lemma \ref{Geodesic and 1 Lem}, $\Al\neq\iota\Al$. Therefore, $\Al=\Be,\iota\Be$. Thus $\wt{\Al}=\wt{\Be}$

\noindent
($(2)\Rightarrow (1)$) : By the definition of $\sim$, $\Al=\Be, \iota\Be$. Therefore, $\ut{\Al}=\pm\ut{\Be}$.
\end{proof}

\begin{thm} [Annihilator Theorem]
 The annihilator $\mathrm{Ann}_{A_0}{A_1}$ of $A_1$ in $A_0$ is generated by the element of the form $\Al+\iota\Al$ such that $\Al$ is non-essential.
 The annihilator $\mathrm{Ann}_{A_1}{A_i}$ of $A_i$ in $A_1$  $(i=0,1)$ is generated by the element of the form $\Al-\iota\Al$ such that $\Al$ is non-essential. In other words, the annihilator $\mathrm{Ann}_{K\hat{\pi}}{A_i}$ of $A_i$ in $K\hat{\pi}$ $(i=0,1)$ is generated by the classes of non-essential curves as a $K$-module.
\end{thm}

The proof of the Annihilator Theorem is based on some modification of Lemma \ref{Lem 3.5}. So, we will recall the proof of Lemma \ref{Lem 3.5} before proving the Annihilator Theorem.\\

\noindent
\textbf{Proof of the Lemma 3.9.} \cite[Lemma 3.5]{Chas-Kabiraj}

\begin{proof}
For any $m \in \mathbb{N}$,
\begin{align*}
[\wt{\Al^m}, \wt{\Be}]_{TWG} &= \sum_{i=1}^k c_i[\wt{\Al^m}, \wt{\Be}_i]_{TWG} \\
&= m \sum_{i=1}^k c_i\sum_{P \in \Al\cap \Be_i}\wampbiz - \wampbii.
\end{align*}

Fix a metric $X$. We take $\Al(X)$, $\Be_1(X), \ldots, \Be_k(X)$ as representatives of $\Al$, $\Be_1, \ldots, \Be_k$. For each $m$, the sum
\begin{align*}
\sum_{i=1}^k c_i\sum_{P \in \Al\cap \Be_i}\wampbiz - \wampbii
\end{align*}
has $I = 2(i(\Al,\Be_1) + \ldots + i(\Al,\Be_k))$ terms before performing possible cancellations. Consider $P$, one of the intersection points of $\Al(X)$ and $\bigcup_{i=1}^k \Be_i(X)$, and choose one of the terms of the bracket associated with $P$, $0$ or $\infty$. For simplicity, assume that the chosen term is $\wt{(\alpha^m*_{P}\beta_1)}_{0}$.

If for all $1 \leq m \leq I + 1$, we have that $[\wt{\Al^m}, \wt{\Be}]_{TWG} = 0$, then there exist $m_1, m_2 \in \{1, 2, \ldots, I + 1\}$, and an $X$-geodesic $\Be_i$, such that one of the following holds:
\begin{enumerate}
  \item[(a)] $\wt{(\alpha^m*_{P}\beta_1)}_{0} = \wt{(\alpha^m*_{Q}\beta_i)}_{0}$ for $m = m_1$ and $m = m_2$.
  \item[(b)] $\wt{(\alpha^m*_{P}\beta_1)}_{0} = \wt{(\alpha^m*_{Q}\beta_i)}_\infty$ and $m = m_1$ and $m = m_2$.
\end{enumerate}

\noindent
By Lemma \ref{Prop 3.4} (1), (a) is not possible. Hence (b) holds, and by Lemma \ref{Prop 3.4} (2), there exists a point $R$ in $\Al(X)\cap \Be_1(X)$ such that the angle $\theta_{P}(X)$ is strictly smaller than the angle $\theta_{R}(X)$. This is not possible because $P$ was chosen arbitrarily, and so the proof is complete.
\end{proof}

\begin{lemma} [modified version of Lemma 3.9] \label{modified version of Lemma 3.9}
Let $\Al, \Be_1, \ldots, \Be_k, \iota\Al, \iota\Be_1, \ldots, \iota\Be_k$ be pairwise distinct free homotopy classes such that $\Al$ is simple. Let $\Be = \sum_{i=1}^{k} c_i \Be_i$ where the coefficients $c_1, c_2, \ldots, c_k$ are nonzero elements of $K$. If there exists $j \in \{1, \ldots, k\}$ such that $i(\Al, \Be_j) > 0$ , then the following holds:

\begin{enumerate}
 \item[(1)]there exists a positive integer $m_0$ such that $[\wt{\Al^m}, \wt{\Be}]_{TWG} \neq 0$ for all $m \geq m_0$.
 \item[(2)]there exists a positive integer $m_0$ such that $[\wt{\Al^m}, \ut{\Be}]_{TWG} \neq 0$ for all $m \geq m_0$.
 \item[(3)]there exists a positive integer $m_0$ such that $[\ut{\Al^m}, \wt{\Be}]_{TWG} \neq 0$ for all $m \geq m_0$.
 \item[(4)]there exists a positive integer $m_0$ such that $[\ut{\Al^m}, \ut{\Be}]_{TWG} \neq 0$ for all $m \geq m_0$.
\end{enumerate}

\end{lemma}

Note that (1) is Lemma \ref{Lem 3.5} itself.

\begin{proof}
First, we will show (4). For any $m \in \mathbb{N}$,
\begin{align*}
[\ut{\Al^m}, \ut{\Be}]_{TWG} &= \sum_{i=1}^k c_i[\ut{\Al^m}, \ut{\Be}_i]_{TWG} \\
&= m \sum_{i=1}^k c_i\sum_{P \in \Al\cap \Be_i}\varepsilon_{P}(\Al^m,\Be_i)(\wampbiz + \wampbii).
\end{align*}
Since $\varepsilon_{P}(\Al^m,\Be_i)$ is at most $1$ or $-1$, the same argument as in the proof of Lemma \ref{Lem 3.5} works.

Next, we will show (3). For any $m \in \mathbb{N}$,
\begin{align*}
[\ut{\Al^m}, \wt{\Be}]_{TWG} &= \sum_{i=1}^k c_i[\ut{\Al^m}, \wt{\Be}_i]_{TWG} \\
&= m \sum_{i=1}^k c_i\sum_{P \in \Al\cap \Be_i}\uampbiz - \uampbii.
\end{align*}

Fix a metric $X$. We take $\Al(X)$, $\Be_1(X), \ldots, \Be_k(X)$ as representatives of $\Al$, $\Be_1, \ldots, \Be_k$. For each $m$, the sum
\begin{align*}
\sum_{i=1}^k c_i\sum_{P \in \Al\cap \Be_i}\uampbiz - \uampbii
\end{align*}
has $I = 2(i(\Al,\Be_1) + \ldots + i(\Al,\Be_k))$ terms before performing possible cancellations. Consider $P$, one of the intersection points of $\Al(X)$ and $\bigcup_{i=1}^k \Be_i(X)$, and choose one of the terms of the bracket associated with $P$, $0$ or $\infty$. For simplicity, assume that the chosen term is $\ut{(\alpha^m*_{P}\beta_1)}_{0}$.

If for all $1 \leq m \leq I + 1$, we have that $[\ut{\Al^m}, \wt{\Be}]_{TWG} = 0$, then there exist $m_1, m_2 \in \{1, 2, \ldots, I + 1\}$, and an $X$-geodesic $\Be_i$, such that one of the following holds:
\begin{enumerate}
  \item[(c)] $\ut{(\alpha^m*_{P}\beta_1)}_{0} = \pm\ut{(\alpha^m*_{Q}\beta_i)}_{0}$ for $m = m_1$ and $m = m_2$.
  \item[(d)] $\ut{(\alpha^m*_{P}\beta_1)}_{0} = \pm\ut{(\alpha^m*_{Q}\beta_i)}_\infty$ and $m = m_1$ and $m = m_2$.
\end{enumerate}
Here, we note that the term $\ut{\Al}$ may cancel the term $-\ut{\Be}$ in $K\ut{\pi}$. By Lemma \ref{Key Lem}, (c) is equivalent to (a) and (d) is equivalent to (b). Therefore, the same argument as in the proof of Lemma \ref{Lem 3.5} works.

For (2), the proof is exactly the same as (3). So the proof is complete.
\end{proof}

\noindent
\textbf{Proof of the Annihilator Theorem.}

\begin{proof}
 The Annihilator Theorem follows from Lemma \ref{modified version of Lemma 3.9} and Lemma \ref{Lem 3.6}.
\end{proof}

Theorems similar to the Annihilator Theorem hold for the universal enveloping algebra and the symmetric algebra of $K\hat{\pi}$. We will prove them in the next section.

\section{universal enveloping algebra and symmetric algebra of the goldman lie algebra}

Denote by $U(K\hat{\pi})$ the universal enveloping algebra, and by $S(K\hat{\pi})$ the symmetric algebra of $K\hat{\pi}$. The algebra $U(K\hat{\pi})$ has a natural Poisson algebra structure with the commutator being the Lie bracket. The algebra $S(K\hat{\pi})$ is also a Poisson algebra by extending the Lie bracket of $K\hat{\pi}$ using the Leibniz rule.

In \cite{Turaev1991}, Turaev introduced a canonical deformation of $S(K\hat{\pi})$  depending on one parameter $k \in K$ Namely, for $\Al, \Be \in \hat{\pi}$, set
\begin{align*}
[\Al, \Be]_{G,k} = [\Al, \Be]_{G} - k(\Al, \Be)\Al\Be
\end{align*}
where $(\Al, \Be)$ is the algebraic intersection number of $\Al$ and $\Be$, and $\Al\Be$ is the product of $\Al$ and $\Be$ in $S(K\hat{\pi})$. By the Leibniz rule,  the bracket $[-, -]_{G,k} $ is extended to a pairing $[-, -]_{G,k} : S(K\hat{\pi}) \times S(K\hat{\pi}) \to S(K\hat{\pi})$. The algebra $S(K\hat{\pi})$ equipped with the bracket $[-, -]_{G,k} $ is also a Poisson algebra. It is denoted by $S_{k}(K\hat{\pi})$. In particular, the algebra $S_{0}(K\hat{\pi})$ is the ordinary symmetric algebra $S(K\hat{\pi})$.

By the isomorphism from $K\hat{\pi}$ to $K\wt{\pi}\oplus K\ut{\pi}$, we can induce a Poisson algebra structure of $S_{k}(K\hat{\pi})$ on $S_{k}(K\wt{\pi}\oplus K\ut{\pi})$. We denote by $[-, -]_{TWG,k}$ the Poisson bracket on $S_{k}(K\wt{\pi}\oplus K\ut{\pi})$. For $\Al, \Be \in \hat{\pi}$, the Poisson bracket $[\wt{\alpha},\wt{\beta}]_{TWG,k}$, $[\wt{\alpha},\ut{\beta}]_{TWG,k}$, $[\ut{\alpha},\wt{\beta}]_{TWG,k}$ and $[\ut{\alpha},\ut{\beta}]_{TWG,k}$ can be written as:

\begin{align*}
&[\wt{\Al}, \wt{\Be}]_{TWG,k} = [\wt{\Al}, \wt{\Be}]_{TWG} - k(\Al, \Be)\ut{\Al}\ut{\Be},\\
&[\wt{\Al}, \ut{\Be}]_{TWG,k} = [\wt{\Al}, \ut{\Be}]_{TWG} - k(\Al, \Be)\ut{\Al}\wt{\Be},\\
&[\ut{\Al}, \wt{\Be}]_{TWG,k} = [\ut{\Al}, \wt{\Be}]_{TWG} - k(\Al, \Be)\wt{\Al}\ut{\Be},\\
&[\ut{\Al}, \ut{\Be}]_{TWG,k} = [\ut{\Al}, \ut{\Be}]_{TWG} - k(\Al, \Be)\wt{\Al}\wt{\Be}.
\end{align*}

There are canonical maps from $K\hat{\pi}$ to $U(K\hat{\pi})$ and $S(K\hat{\pi})$. To simplify notation, we denote an element and its image under these maps by the same notation. We also denote the product of two elements in $U(K\hat{\pi})$ and $S(K\hat{\pi})$ simply by juxtaposing the elements.
Let us recall the Poincaré-Birkhoff-Witt theorem for $K\hat{\pi}$. 

Let us consider the sets
\begin{align*}
&\wt{T} = \{\wt{x} \in K\wt{\pi}| x \in \hat{\pi} \},\\
&\ut{T} = \{\ut{x} \in K\ut{\pi}| x \in \hat{\pi} \setminus \{{\bf 1}\}\},\\
&T = \wt{T}\cup\ut{T} .
\end{align*}
\noindent
The set $\wt{T}$ (resp. $\ut{T}$) gives a basis of $K\wt{\pi}$ (resp. $K\ut{\pi}$). Therefore, the set $T$ gives a basis of $K\hat{\pi}$ as a $K$-module. Furthermore, we fix a total order $\leq$ on the set $\wt{T}$ and $\ut{T}$. By defining every element of $\wt{T}$ is less than every element of $\ut{T}$, we also define a total order $\leq$ on the set $T$. Consider the set
\begin{multline} \label{Basis wt ut by PBW}
\mathcal{B}=\lbrace \wt{\Be_{i_1}}\ldots\wt{\Be_{i_h}}\ut{\Be_{i_{h+1}}}\ldots\ut{\Be_{i_{h+l}}}| \wt{\Be_{i_1}},\ldots,\wt{\Be_{i_h}} \in \wt{T}, \ut{\Be_{i_{h+1}}},\ldots,\ut{\Be_{i_{h+l}}}\in \ut{T}, h,l\in \Z_{\geq 0}\\, 
\wt{\Be_{i_1}}\leq\ldots\leq \wt{\Be_{i_h}}, \ut{\Be_{i_{h+1}}}\leq\ldots\leq\ut{\Be_{i_{h+l}}}\rbrace.
\end{multline}

\noindent
Both $U(K\hat{\pi})$ and $S(K\hat{\pi})$ are freely generated by $\mathcal{B}$ as a $K$-module. Moreover, the natural maps from $K\hat{\pi}$ into $U(K\hat{\pi})$ and $S(K\hat{\pi})$ are injective Lie algebra homomorphisms.

On the other hand, if $k\neq0$, the natural map from $K\hat{\pi}$ into $S_{k}(K\hat{\pi})$ is not a Lie algebra homomorphism, but an injective linear map.

\begin{thm}
The annihilator $\mathrm{Ann}_{S_{k}(K\hat{\pi})}{A_i}$ of $A_i$ in $S_{k}(K\hat{\pi})$ $(i=0,1)$ is generated by scalars $K$ and the classes of non-essential curves as a $K$-algebra.
\end{thm}

\begin{proof}
For $i=0$, we compute the annihilator of $A_i$ in $S_{k}(K\hat{\pi})$. For $i=1$, the proof is exactly the same.
Let $\Be$ be an element of the annihilator of $A_0$ in $S_{k}(K\hat{\pi})$. Then by (\ref{Basis wt ut by PBW}),
\begin{align*}
\Be=\sum_{i=1}^n c_i\wt{\Be_{i_1}}\ldots\wt{\Be_{i_{h_i}}}\ut{\Be_{i_{{h_i}+1}}}\ldots\ut{\Be_{i_{{h_i}+{l_i}}}}
\end{align*}

\noindent
where $\wt{\Be_{i_1}},\ldots,\wt{\Be_{i_{h_i}}} \in \wt{T}, \ut{\Be_{i_{{h_i}+1}}},\ldots,\ut{\Be_{i_{{h_i}+{l_i}}}}\in \ut{T}, h,l\in \Z_{\geq 0}, \wt{\Be_{i_1}}\leq\ldots\leq \wt{\Be_{i_{h_i}}}, \ut{\Be_{i_{{h_i}+1}}}\leq\ldots\leq\ut{\Be_{i_{{h_i}+{l_i}}}}, c_i\neq0 $ for all $i \in \lbrace1, 2, . . . , m\rbrace$. We prove each free homotopy class $\Be_{i_{j}}$ is non-essential.  We have
\begin{align*}
0&=[\wt{\Al^m}, \Be]_{TWG,k}\\
&= \sum_{i=1}^n c_i[\wt{\Al^m}, \wt{\Be_{i_1}}\ldots\wt{\Be_{i_{h_i}}}\ut{\Be_{i_{{h_i}+1}}}\ldots\ut{\Be_{i_{{h_i}+{l_i}}}}]_{TWG,k}\\
\begin{split}
&= \sum_{i=1}^n c_i\biggl(\sum_{j=1}^{h_i} \wt{\Be_{i_1}}\ldots\wt{\Be_{i_{j-1}}}\bigl([\wt{\Al^m}, \wt{\Be_{i_j}}]_{TWG}-k(\Al^m, \Be_{i_j})\ut{\Al^m}\ut{\Be_{i_j}}\big)\wt{\Be_{i_{j+1}}}\ldots\wt{\Be_{i_{h_i}}}\ut{\Be_{i_{{h_i}+1}}}\ldots\ut{\Be_{i_{{h_i}+{l_i}}}}\\
& \quad\quad\quad\quad+ \sum_{j={h_i}+1}^{{h_i}+{l_i}} \wt{\Be_{i_1}}\ldots\wt{\Be_{i_{h_i}}}\ut{\Be_{i_{{h_i}+1}}}\ldots\ut{\Be_{i_{j-1}}}\bigl([\wt{\Al^m}, \ut{\Be_{i_j}}]_{TWG}-k(\Al^m, \Be_{i_j})\ut{\Al^m}\wt{\Be_{i_j}}\bigr)\ut{\Be_{i_{j+1}}}\ldots\ut{\Be_{i_{{h_i}+{l_i}}}}\biggr)\\
&= \sum_{i=1}^n c_i\biggl(\sum_{j=1}^{h_i} m\Bigl(\sum_{P_{i_j} \in \Al\cap \Be_{i_j}}\wt{\Be_{i_1}}\ldots\wt{\Be_{i_{j-1}}}\bigl(\wt{(\Al^m*_{P_{i_j}}\Be_{i_j})}_{0}-\wt{(\Al^m*_{P_{i_j}}\Be_{i_j})}_{\infty}-k(\Al, \Be_{i_j})\ut{\Al^m}\ut{\Be_{i_j}}\bigr)\\
& \quad\quad\quad\quad\quad\quad\quad\quad\quad\quad\quad\quad\quad\quad\quad\quad\quad\quad\quad\quad
\quad\quad\quad\quad\quad\quad\quad\quad\wt{\Be_{i_{j+1}}}\ldots\wt{\Be_{i_{h_i}}}\ut{\Be_{i_{{h_i}+1}}}\ldots\ut{\Be_{i_{{h_i}+{l_i}}}}\Bigr)\\
& \quad\quad\quad\quad+ \sum_{j={h_i}+1}^{{h_i}+{l_i}} m\Bigl(\sum_{P_{i_j} \in \Al\cap \Be_{i_j}}\wt{\Be_{i_1}}\ldots\wt{\Be_{i_{h_i}}}\ut{\Be_{i_{{h_i}+1}}}\ldots\ut{\Be_{i_{j-1}}}\Bigl(\varepsilon_{P_{i_j}}(\Al^m,\Be_{i_j})\bigl(\ut{(\Al^m*_{P_{i_j}}\Be_{i_j})}_{0}\\
&\quad\quad\quad\quad\quad\quad\quad\quad\quad\quad
\quad\quad\quad\quad\quad\quad\quad+\ut{(\Al^m*_{P_{i_j}}\Be_{i_j})}_{\infty}\bigr)-k(\Al^m, \Be_{i_j})\ut{\Al^m}\wt{\Be_{i_j}}\Bigr)\ut{\Be_{i_{j+1}}}\ldots\ut{\Be_{i_{{h_i}+{l_i}}}}\Bigr)\biggr)\\
\end{split}
\end{align*}

Fix a metric $X \in Teich(S)$ and take $\Al$ to be any simple $X$-geodesic. Fix any $\Be_{i_j}$ and consider $P \in \Al(X)\cap\Be_{i_j}(X)$ such that $\theta_P(X) \geq \theta_R(X)$ for all $R \in \Al(X)\cap\Be_{i_j}(X)$.


First, we will consider in the case where $1\le j \le {h_i}$. As $[\wt\Al^m, \Be]_{TWG,k}= 0$ for all positive integer $m$, from the above expression, there exists $\wt\Be_{i'_{j'}} \neq \wt\Be_{i_j}$ such that for infinitely many positive integer $m$, one of the following is true.\\
• $\wt{(\Al^m*_{P}\Be_{i_j})}_{0} = \wt{(\Al^m*_{Q}\Be_{i'_{j'}})}_{0}$ for some $Q \in \Al(X) \cap \Be_{i'_{j'}}(X)$. This is impossible by Lemma \ref{Prop 3.4}.\\
• $\wt{(\Al^m*_{P}\Be_{i_j})}_{0} = \wt{(\Al^m*_{Q}\Be_{i'_{j'}})}_{\infty}$ for some $Q \in \Al(X) \cap \Be_{i'_{j'}}(X)$. This is impossible by Lemma \ref{Prop 3.4}.\\
•  $\wt{(\Al^m*_{P}\Be_{i_j})}_{0} = \wt{\Be_{i'_{j'}}}$. This is impossible by Lemma \ref{Lem for type3}.\\
•  $\wt{(\Al^m*_{P}\Be_{i_j})}_{0} = \wt{\Al^m}$. This is impossible by Lemma \ref{Lem for type4}.

\noindent
Thus, each $\Be_{i_j}$ is disjoint from $\Al$. As $\Al$ is an arbitrary simple closed geodesic, each $\Be_{i_j}$ is disjoint from every simple closed geodesic on the surface. Hence by Lemma \ref{Lem 3.6}, each $\Be_{i_j}$ is non-essential in the case where $1\le j \le {h_i}$.

Next, we will consider in the case where ${h_i}+1\le j \le {h_i}+{l_i}$. As $[\wt\Al^m, \Be]_{TWG,k}= 0$ for all positive integer $m$, from the above expression and, there exists $\ut\Be_{i'_{j'}} \neq \ut\Be_{i_j}$ such that for infinitely many positive integer $m$, one of the following is true.\\
• $\ut{(\Al^m*_{P}\Be_{i_j})}_{0} = \pm\ut{(\Al^m*_{Q}\Be_{i'_{j'}})}_{0}$ for some $Q \in \Al(X) \cap \Be_{i'_{j'}}(X)$. This is impossible by Lemma \ref{Key Lem} and Lemma \ref{Prop 3.4}.\\
• $\ut{(\Al^m*_{P}\Be_{i_j})}_{0} = \pm\ut{(\Al^m*_{Q}\Be_{i'_{j'}})}_{\infty}$ for some $Q \in \Al(X) \cap \Be_{i'_{j'}}(X)$. This is impossible by Lemma \ref{Key Lem} and Lemma \ref{Prop 3.4}.\\
•  $\ut{(\Al^m*_{P}\Be_{i_j})}_{0} = \pm\ut{\Be_{i'_{j'}}}$. This is impossible by Lemma \ref{Key Lem} and Lemma \ref{Lem for type3}.\\
•  $\ut{(\Al^m*_{P}\Be_{i_j})}_{0} = \pm\ut{\Al^m}$. This is impossible by Lemma \ref{Key Lem} and Lemma \ref{Lem for type4}.

\noindent
For the same reason as before, each $\Be_{i_j}$ is non-essential in the case where ${h_i}+1\le j \le {h_i}+{l_i}$. 
Therefore, each free homotopy class $\Be_{i_{j}}$ is non-essential. 
\end{proof}
\begin{thm}
The annihilator $\mathrm{Ann}_{U(K\hat{\pi})}{A_i}$ of $A_i$ in $U(K\hat{\pi})$ $(i=0,1)$ is generated by scalars $K$ and the classes of non-essential curves as a $K$-algebra.
\end{thm}
\begin{proof}
For $i=0$, we compute the annihilator of $A_i$ in $U(K\hat{\pi})$. For $i=1$, the proof is exactly the same. Let $\Be$ be an element of the annihilator of $A_0$ in $U(K\hat{\pi})$. Then by (\ref{Basis wt ut by PBW}),
\begin{align*}
\Be=\sum_{i=1}^n c_i\wt{\Be_{i_1}}\ldots\wt{\Be_{i_{h_i}}}\ut{\Be_{i_{{h_i}+1}}}\ldots\ut{\Be_{i_{{h_i}+{l_i}}}}
\end{align*}

\noindent
where $\wt{\Be_{i_1}},\ldots,\wt{\Be_{i_{h_i}}} \in \wt{T}, \ut{\Be_{i_{{h_i}+1}}},\ldots,\ut{\Be_{i_{{h_i}+{l_i}}}}\in \ut{T}, h,l\in \Z_{\geq 0}, \wt{\Be_{i_1}}\leq\ldots\leq \wt{\Be_{i_{h_i}}}, \ut{\Be_{i_{{h_i}+1}}}\leq\ldots\leq\ut{\Be_{i_{{h_i}+{l_i}}}}, c_i\neq0 $ for all $i \in \lbrace1, 2, . . . , m\rbrace$. We prove each free homotopy class $\Be_{i_{j}}$ is non-essential. We have
\begin{align*}
0&=[\wt{\Al^m}, \Be]_{TWG}\\
&= \sum_{i=1}^n c_i[\ut{\Al^m}, \wt{\Be_{i_1}}\ldots\wt{\Be_{i_{h_i}}}\ut{\Be_{i_{{h_i}+1}}}\ldots\ut{\Be_{i_{{h_i}+{l_i}}}}]_{TWG}\\
\begin{split}
&= \sum_{i=1}^n c_i\biggl(\sum_{j=1}^{h_i} \wt{\Be_{i_1}}\ldots\wt{\Be_{i_{j-1}}}\bigl([\wt{\Al^m}, \wt{\Be_{i_j}}]_{TWG}\bigr)\wt{\Be_{i_{j+1}}}\ldots\wt{\Be_{i_{h_i}}}\ut{\Be_{i_{{h_i}+1}}}\ldots\ut{\Be_{i_{{h_i}+{l_i}}}}\\
& \quad\quad\quad+ \sum_{j={h_i}+1}^{{h_i}+{l_i}} \wt{\Be_{i_1}}\ldots\wt{\Be_{i_{h_i}}}\ut{\Be_{i_{{h_i}+1}}}\ldots\ut{\Be_{i_{j-1}}}\bigl([\wt{\Al^m}, \ut{\Be_{i_j}}]_{TWG}\bigr)\ut{\Be_{i_{j+1}}}\ldots\ut{\Be_{i_{{h_i}+{l_i}}}}\biggr)\\
&= \sum_{i=1}^n\biggl(\sum_{j=1}^{h_i} m\Bigl(\sum_{P_{i_j} \in \Al\cap \Be_{i_j}}\wt{\Be_{i_1}}\ldots\wt{\Be_{i_{j-1}}}\wt{(\Al^m*_{P_{i_j}}\Be_{i_j})}_{0}\wt{\Be_{i_{j+1}}}\ldots\wt{\Be_{i_{h_i}}}\ut{\Be_{i_{{h_i}+1}}}\ldots\ut{\Be_{i_{{h_i}+{l_i}}}}\Bigr)\\
& \quad\quad- \sum_{j=1}^{h_i} m\Bigl(\sum_{P_{i_j} \in \Al\cap \Be_{i_j}}\wt{\Be_{i_1}}\ldots\wt{\Be_{i_{j-1}}}\wt{(\Al^m*_{P_{i_j}}\Be_{i_j})}_{\infty}\wt{\Be_{i_{j+1}}}\ldots\wt{\Be_{i_{h_i}}}\ut{\Be_{i_{{h_i}+1}}}\ldots\ut{\Be_{i_{{h_i}+{l_i}}}}\Bigr)\\
& \quad\quad+ \sum_{j={h_i}+1}^{{h_i}+{l_i}} m\Bigl(\sum_{P_{i_j} \in \Al\cap \Be_{i_j}}\varepsilon_{P_{i_j}}(\Al^m,\Be_{i_j})\wt{\Be_{i_1}}\ldots\wt{\Be_{i_{h_i}}}\ut{\Be_{i_{{h_i}+1}}}\ldots\ut{\Be_{i_{j-1}}}\ut{(\Al^m*_{P_{i_j}}\Be_{i_j})}_{0}\ut{\Be_{i_{j+1}}}\ldots\ut{\Be_{i_{{h_i}+{l_i}}}}\Bigr)\\
& \quad\quad+ \sum_{j={h_i}+1}^{{h_i}+{l_i}} m\Bigl(\sum_{P_{i_j} \in \Al\cap \Be_{i_j}}\varepsilon_{P_{i_j}}(\Al^m,\Be_{i_j})\wt{\Be_{i_1}}\ldots\wt{\Be_{i_{h_i}}}\ut{\Be_{i_{{h_i}+1}}}\ldots\ut{\Be_{i_{j-1}}}\ut{(\Al^m*_{P_{i_j}}\Be_{i_j})}_{\infty}\ut{\Be_{i_{j+1}}}\ldots\ut{\Be_{i_{{h_i}+{l_i}}}}\Bigr)\biggr)\\
\end{split}
\end{align*}

Fix a metric $X \in Teich(S)$ and take $\Al$ to be any simple $X$-geodesic. By using the relation $xy=yx-[x,y]$, we can rearrange the terms
\begin{align*}
\lbrace \wt{\Be_{i_1}},\ldots,\wt{\Be_{i_{j-1}}},\wt{(\Al^m*_{P_{i_j}}\Be_{i_j})}_{0},\wt{\Be_{i_{j+1}}},\ldots,\wt{\Be_{i_{h_i}}},\ut{\Be_{i_{{h_i}+1}}},\ldots,\ut{\Be_{i_{{h_i}+{l_i}}}}\rbrace
\end{align*}

\noindent
in ascending order with respect to $\leq$. Assume they are elements of the set $\mathcal{B}$ of (\ref{Basis wt ut by PBW}). We also rearrange the terms 
\begin{align*}
\lbrace \wt{\Be_{i_1}},\ldots,\wt{\Be_{i_{j-1}}},\wt{(\Al^m*_{P_{i_j}}\Be_{i_j})}_{\infty},\wt{\Be_{i_{j+1}}},\ldots,\wt{\Be_{i_{h_i}}},\ut{\Be_{i_{{h_i}+1}}},\ldots,\ut{\Be_{i_{{h_i}+{l_i}}}}\rbrace, \\
\lbrace\wt{\Be_{i_1}},\ldots,\wt{\Be_{i_{h_i}}}\ut{\Be_{i_{{h_i}+1}}},\ldots,\ut{\Be_{i_{j-1}}},\ut{(\Al^m*_{P_{i_j}}\Be_{i_j})}_{0},\ut{\Be_{i_{j+1}}},\ldots,\ut{\Be_{i_{{h_i}+{l_i}}}}\rbrace,\\
\lbrace\wt{\Be_{i_1}},\ldots,\wt{\Be_{i_{h_i}}}\ut{\Be_{i_{{h_i}+1}}},\ldots,\ut{\Be_{i_{j-1}}},\ut{(\Al^m*_{P_{i_j}}\Be_{i_j})}_{\infty},\ut{\Be_{i_{j+1}}},\ldots,\ut{\Be_{i_{{h_i}+{l_i}}}}\rbrace.
\end{align*}

\noindent
in ascending order with respect to $\leq$. Assume they are elements of the set $\mathcal{B}$ of (\ref{Basis wt ut by PBW}).

We choose one $i$ (which we call $I$) such that $h_i+l_i$ is maximum. For each $j$, we consider $P \in \Al(X)\cap\Be_{I_j}(X)$ such that $\theta_P(X) \geq \theta_R(X)$ for all $R \in \Al(X)\cap\Be_{I_j}(X)$.

First, we will consider in the case where $1\le j \le {h_I}$. As $[\wt\Al^m, \Be]_{TWG}= 0$ for all positive integer $m$, from the above expression, there exists $\wt\Be_{i'_{j'}} \neq \wt\Be_{I_j}$ such that for infinitely many positive integer $m$, one of the following is true.\\
• $\wt{(\Al^m*_{P}\Be_{I_j})}_{0} = \wt{(\Al^m*_{Q}\Be_{i'_{j'}})}_{0}$ for some $Q \in \Al(X) \cap \Be_{i'_{j'}}(X)$. This is impossible by Lemma \ref{Prop 3.4}.\\
• $\wt{(\Al^m*_{P}\Be_{I_j})}_{0} = \wt{(\Al^m*_{Q}\Be_{i'_{j'}})}_{\infty}$ for some $Q \in \Al(X) \cap \Be_{i'_{j'}}(X)$. This is impossible by Lemma \ref{Prop 3.4}.\\
•  $\wt{(\Al^m*_{P}\Be_{I_j})}_{0} = \wt{\Be_{i'_{j'}}}$. This is impossible by Lemma \ref{Lem 2.7}.

\noindent
Thus, each $\Be_{I_j}$ is disjoint from $\Al$. As $\Al$ is an arbitrary simple closed geodesic, each $\Be_{I_j}$ is disjoint from every simple closed geodesic on the surface. Hence by Lemma \ref{Lem 3.6}, each $\Be_{I_j}$ is non-essential in the case where $1\le j \le {h_I}$.

Next, we will consider in the case where ${h_I}+1\le j \le {h_I}+{l_I}$. As $[\wt\Al^m, \Be]_{TWG}= 0$ for all positive integer $m$, from the above expression and, there exists $\ut\Be_{i'_{j'}} \neq \ut\Be_{I_j}$ such that for infinitely many positive integer $m$, one of the following is true.\\
• $\ut{(\Al^m*_{P}\Be_{I_j})}_{0} = \pm\ut{(\Al^m*_{Q}\Be_{i'_{j'}})}_{0}$ for some $Q \in \Al(X) \cap \Be_{i'_{j'}}(X)$. This is impossible by Lemma \ref{Key Lem} and Lemma \ref{Prop 3.4}.\\
• $\ut{(\Al^m*_{P}\Be_{I_j})}_{0} = \pm\ut{(\Al^m*_{Q}\Be_{i'_{j'}})}_{\infty}$ for some $Q \in \Al(X) \cap \Be_{i'_{j'}}(X)$. This is impossible by Lemma \ref{Key Lem} and Lemma \ref{Prop 3.4}.\\
•  $\ut{(\Al^m*_{P}\Be_{I_j})}_{0} = \pm\ut{\Be_{i'_{j'}}}$. This is impossible by Lemma \ref{Key Lem} and Lemma \ref{Lem 2.7}.

\noindent
For the same reason as before, each $\Be_{I_j}$ is non-essential in the case where ${h_I}+1\le j \le {h_I}+{l_I}$. 

Now, an element 

\begin{align*}
\Be-c_I\wt{\Be_{I_1}}\ldots\wt{\Be_{I_{h_I}}}\ut{\Be_{I_{{h_I}+1}}}\ldots\ut{\Be_{I_{{h_I}+{l_I}}}}
\end{align*}

\noindent
also belongs to $\mathrm{Ann}_{U(K\hat{\pi})}{A_0}$. We repeat the above argument until each $\Be_{i_j}$ is non-essential. Here, the proof is complete. 

\end{proof}


\bibliography{master}
\bibliographystyle{plain}

\end{document}